\documentclass{amsart}
\usepackage{pgf}
\usepackage{tikz}
\usetikzlibrary{arrows,automata}
\usepackage[latin1]{inputenc}
\usepackage{mathrsfs}
\usepackage{amssymb}
\usepackage{fancyhdr}

\newtheorem{theorem}{Theorem}
\theoremstyle{plain}

\newtheorem{corollary}{Corollary}

\newtheorem{example}{Example}

\newtheorem{lemma}{Lemma}

\newtheorem{proposition}{Proposition}

\numberwithin{equation}{section}

\begin{document}
\title{On the rational approximations to the real numbers corresponding to the differences of the Fibonacci sequence}
\author{Ying-jun Guo, Zhi-xiong Wen and Jie-Meng Zhang}

%\ead{guoyingjun2005@126.com}
%\address[rvt1] {College of Sciences, Huazhong Agricultural University, Wuhan, 430070, P. R. China}

%\address[rvt2]{Department of Mathematics, Huazhong University of Science and Technology, Wuhan, Hubei, 430074, P. R. China}

\maketitle

\begin{abstract}
Based on the structure of Fibonacci sequence, we give a new proof for the irrationality exponents of the Fibonacci real numbers. Moreover, we obtain all the irrationality exponents of the real numbers corresponding to the differences of Fibonacci sequence.
\end{abstract}    % the abstract

%\begin{keyword}
%irrationality exponent; Fibonacci sequence; difference sequence
%\end{keyword}       % the keywords
\maketitle   % 2010 MR Subject Classification

\section{Introduction}
In 2007, Adamczewski and Allouche \cite{AA07} proved that the irrationality exponent of the Fibonacci real numbers equals $1+\frac{1+\sqrt{5}}{2}$. For this, they studied the real numbers
\begin{equation*}
    S_{b}(\alpha)=(b-1)\sum_{n\geq1}\frac{1}{b^{\lfloor n\alpha \rfloor}}
\end{equation*}
for any irrational $\alpha$ and any integer $b$, both larger than $1$. By the continued fraction expansion \footnote{It was originally discovered by B\"{o}hmer \cite{B26}. It was independently
rediscovered by Danilov \cite{Danilov72}, Davison \cite{Davison77}, Adams and Davison \cite{AD77}, Bullett and Sentenac \cite{BS94}, and Shiu \cite{S99}.} of $S_{b}(\alpha)$, %a nice result of B\"{o}hmer \cite{B26},%Adams and Davison \cite{AD77},
~they proved that the irrationality exponent of the number $S_{b}(\alpha)$ equals $1+\limsup_{n\rightarrow\infty}[a_n,a_{n-1},\cdots,a_0]$ for any irrational number $\alpha=[a_0,a_1,a_2,\cdots]>1$ and any integer $b\geq2$.
Recall that the \emph{irrationality exponent} (sometimes called the \emph{irrationality measure}) of an irrational real number $\xi$, denoted by $\mu(\xi)$, is defined as the supremum of the set of real numbers $\mu$ such
that the inequality
\begin{eqnarray*}
\left|\xi-\frac{p}{q}\right|<\frac{1}{q^{\mu}}
\end{eqnarray*}
holds for infinitely many $\frac{p}{q}\in\mathbb{Q}$. Clearly, $\mu(\xi)\geq2$ for any irrational number $\xi$ and for almost all real numbers, including all the algebraic irrational numbers, $\mu(\xi)=2$.

In \cite{AY07}, Adamczewski and Bugeaud proved that every irrational automatic or binary morphic real number (whose $b$-expansion is an automatic or morphic sequence for some integer $b\geq2$) is transcendental. Then many automatic transcendental numbers have been studied in \cite{AC06,AR09,Bugeaud11,CV12,C12,GWW14}. Moreover, Bugeaud, Krieger and Shallit conjectured in \cite{BKS11} that the irrationality exponent of every automatic(resp.,morphic) number is rational (resp.,algebraic).
In this paper, we are interested in the irrationality exponents of the binary morphic real numbers $\xi_{\varepsilon,b}$ for any integer $b\geq2$, where
\begin{equation*}
 \xi_{\varepsilon,b}:=\sum_{i\geq0}\frac{\varepsilon_{i}}{b^{i}}=\varepsilon_{0}+\frac{\varepsilon_{1}}{b}+\frac{\varepsilon_{2}}{b^{2}}+\cdots,
\end{equation*}
and $\varepsilon=\varepsilon^{k}:=\varepsilon_{0}\varepsilon_{1}\varepsilon_{2}\cdots $ is the fixed point of the morphism
\begin{equation}\label{formula1}
\sigma(0)=0^{k}1,\sigma(1)=0,(k\geq1).
\end{equation}

Note that the morphism $\sigma$ defined by (\ref{formula1}) is an invertible substitution \cite{WW93,WW94,WWW02}. Its fixed point $\varepsilon$ is a Sturmian sequence and studied by many authors \cite{AS03,Tamura99,KTW99,WW94}. In this paper, we will study the structure of $\varepsilon$. Based on the structure, we shall prove following theorem.
\begin{theorem}\label{theorem1}
For any integers $b\geq2,k\geq1$, we have $\mu(\xi_{\varepsilon^{k},b})=1+\frac{k+\sqrt{k^{2}+4}}{2}$.
\end{theorem}
In particular, when $k=1$, the sequence $\varepsilon^{1}$ is the famous Fibonacci sequence. For any integer $b\geq2$, the real number $\xi_{\varepsilon^{1},b}$ is called to be the \emph{Fibonacci real number}. Hence, we have following corollary.
\begin{corollary}\label{c1}
For any integer $b\geq2$, we have $\mu(\xi_{\varepsilon^{1},b})=1+\frac{1+\sqrt{5}}{2}$.
\end{corollary}
In fact, Corollary \ref{c1} gives a new proof for Adamczewski and Allouche's result. To see this, take $\alpha=\frac{1+\sqrt{5}}{2}=[1,1,\cdots]$, define two binary sequences
\begin{equation*}
  g_{\alpha}(n)=\left\{ \begin{array}{ll}
  1, & \hbox{if $n=\lfloor k\alpha \rfloor$ for some integer $k$;} \\
  0, & \hbox{otherwise.}                                                                                                                            \end{array}
   \right.
\end{equation*}
and $f_{\frac{1}{\alpha}}(n)=\lfloor (n+1)\frac{1}{\alpha} \rfloor-\lfloor n\frac{1}{\alpha} \rfloor$. Then, by Lemma 9.1.3 and Corollary 9.1.6 in \cite{AS03}, the sequence $f_{\frac{1}{\alpha}}(n)=1-\varepsilon_{n}$ and
\begin{equation*}
 \sum_{n\geq1}\frac{1}{b^{\lfloor n\alpha \rfloor}}=\sum_{n\geq1}\frac{g_{\alpha}(n)}{b^{n}}=\sum_{n\geq1}\frac{f_{1/\alpha}(n)}{b^{n}}=\sum_{n\geq1}\frac{1-\varepsilon_{n}}{b^{n}}.
\end{equation*}
Hence, $S_{b}(\alpha)=(b-1)(\frac{1}{b-1}-\xi_{\varepsilon^{1},b})=1-(b-1)\xi_{\varepsilon^{1},b}$. Hence, Corollary \ref{c1} implies the following theorem.
\begin{theorem}[Adamczewski and Allouche \cite{AA07}]
If $\alpha=\frac{1+\sqrt{5}}{2}=[1,1,\cdots]$, then for any integer $b\geq2$, $\mu(S_{b}(\alpha))=1+\frac{1+\sqrt{5}}{2}$.
\end{theorem}

Not only we obtain the irrationality exponents corresponding to the sequence $\varepsilon$, but also we give a new property of the irrationality exponent corresponding to the Sturmian sequence. We prove that the irrationality exponent is invariant under direct product with a shift of the original Sturmian sequence.

\begin{proposition}\label{p1}
For any Sturmian sequence $\mathbf{u}\in\{0,1\}^{\mathbb{N}}$, assume the sequence $\mathbf{u}\times S(\mathbf{u}):=\{[u_{i},u_{i+1}]\}_{i\geq0}$ is a non-eventually periodic integer sequence. If their exists an integer $a\geq2$ such that $\mu(\xi_{\mathbf{u},b})=a$ for any $b\geq2$, then $\mu(\xi_{\mathbf{u}\times S(\mathbf{u}),b})=a$ for any integer $b\geq2$.
\end{proposition}

Recently,  %by the invariant property of the irrationality exponent under the uniform morphism,
Guo and Wen  proved in \cite{GW14} that all the irrationality exponents of the differences are equal to 2. Recall that the difference of an binary sequence $\mathbf{u}=u_{0}u_{1}u_{2}\cdots$, denoted by  $\Delta(\mathbf{u})$, is the sequence $\{(u_{n+1}-u_{n})(\bmod 2)\}_{n\ge 0}$. For any $k\geq1$, $\Delta^{k}=\Delta(\Delta^{k-1}(\mathbf{u}))$ and $\Delta^{0}(\mathbf{u})=\mathbf{u}$.
 Using Proposition \ref{p1}, we will obtain all the irrationality exponents of the real numbers corresponding to the differences of Sturmian sequences.
%For any sequence $\mathbf{u}=u_{0}u_{1}u_{2}\cdots\in\{0,1\}^{\mathbb{N}}$, define the difference map $\Delta(\mathbf{u})$ to be the sequence $\{(u_{n+1}-u_{n})(\bmod 2)\}_{n\ge 0}$. Inductively, for $k\geq1$, $\Delta^{k}=\Delta(\Delta^{k-1}(\mathbf{u}))$, where the sequence $\Delta^{0}(\mathbf{u})=\mathbf{u}$.
\begin{corollary}\label{c2}
For any Sturmian sequence $\mathbf{u}\in\{0,1\}^{\mathbb{N}}$, if their exists an integer $a\geq2$ such that $\mu(\xi_{\mathbf{u},b})=a$ for any $b\geq2$, then for any integer $k\geq1$, $\mu(\xi_{\Delta^{k}(\mathbf{u}),b})=a$ for any integer $b\geq2$.
\end{corollary}
In particular, we have
\begin{corollary}\label{c3}
For any integer $b\geq2$, $\mu(\xi_{\Delta^{k}(\mathbf{\varepsilon^{1}}),b})=1+\frac{1+\sqrt{5}}{2}$.
\end{corollary}
Corollary \ref{c2} and Corollary \ref{c3} tell us that we can get the irrationality exponents of a large class of morphis numbers. Moreover, most of them are the numbers whose $b$-expansions are not Sturmian sequences.
\begin{example}
It easy to check that $\Delta^{2}(\varepsilon^{1})=01100011011\cdots$. The sequence $\Delta^{2}(\varepsilon^{1})$ is not a sturmian sequence, since both the $'00'$ block and $'11'$ block appear in the sequence. By Corollary \ref{c3}, we have $\mu(\xi_{\Delta^{2}(\varepsilon^{1}),b})=1+\frac{1+\sqrt{5}}{2}$.
\end{example}

This paper is organized as follows. In Section 2, we study the structure of the sequence $\varepsilon$. In
Section 3, we prove Theorem \ref{theorem1} and Proposition \ref{p1}.
\section{Structure of the sequence $\varepsilon$}
%The structure of Fibonacci sequence have been researched by many authors. Wen and Wen given some properties of the singular words of the Fibonacci word in \cite{WW94,KTW99}.

In this section, we will study the sequence $\varepsilon$.
Set $U_{n}=\sigma^{n}(0)$, where $\sigma^{n}$ denote the $n$-th iteration of $\sigma$ by $\sigma^{n}=\sigma(\sigma^{n-1}),n\geq1$. $U_{n}^{*}$ is denoted the word obtained by interchange the last two letters of $U_{n}$. Define  $f_{n}:=|U_{n}|~(n\geq0)$, where $|W|$ denotes the length of the finite word $W$. Then $(f_{n})_{n\geq-2}$ can be defined by a recursive formula as follow:
\begin{equation}\label{formula2}
f_{n+2}=k f_{n+1}+f_{n},(n\geq-2),f_{-2}=1-k,f_{-1}=1.
\end{equation}

\begin{lemma}\label{lemma1}
For any integer $n\geq2$, we have following statements:

(1) $U_{n-1}U_{n}=U_{n}U_{n-1}^{*}$,

(2) $U_{n+2}=U_{n}(U_{n}^{k}U_{n-1}^{*})^{k}$.

\end{lemma}

\begin{proof} (1) By induction on $n$, the case $n=2$ is trivial.
Assume that the result is true for all $n\leq m$; we need to prove it for $n=m+1$.

Hence, we have
\begin{eqnarray*}
U_{m}U_{m+1}&=&U_{m}U_{m}^{k}U_{m-1}=U_{m}^{k}U_{m}U_{m-1}\\&=&U_{m}^{k}U_{m-1}^{k}U_{m-2}U_{m-1}
=U_{m}^{k}U_{m-1}U_{m-1}^{k-1}U_{m-2}U_{m-1}\\&=&U_{m+1}U_{m-1}^{k-1}U_{m-1}U_{m-2}^{*}=U_{m+1}U_{m-1}^{k}U_{m-2}^{*}=U_{m+1}U_{m}^{*}.
\end{eqnarray*}
which completes the proof.

(2) This follows immediately from (1) and $U_{n+2}=U_{n+1}^{k}U_{n}$ for $n\geq2$.
\end{proof}

\begin{corollary}\label{corollary1}
$~\forall ~n\geq 2$, $\varepsilon_{i}=\varepsilon_{i+f_{n}},$ for $0\leq i\leq f_{n+1}-3$.
\end{corollary}

Now we turn to introduce another numeration system (more details in \cite{AS03}), based on the numbers defined in (\ref{formula2}).
We state the following theorem without proof.

\begin{theorem}\label{th1}
Every integer $n\in\mathbb{N}$ can be uniquely expressed as $n=\sum_{i\geq0}\tau_{i}(n)f_{i}$, with $\tau_{i}(n)\in{\{0,1,\cdots,k\}}$, and if $\tau_{i+1}(n)=k$, then $\tau_{i}(n)=0$, for all $ i\geq0$.
\end{theorem}

The expression is called \emph{regular expression} based on the integer sequence $(f_{n})_{n\geq0}$ defined in (\ref{formula2}).
Let $m,n\in\mathbb{N}$, $j$ be a positive integer, we define
\begin{eqnarray*}
m \equiv_{j} n ,
\end{eqnarray*}
if  $\tau_{i}(m)=\tau_{i}(n)$ , for all $i<j$.

\begin{lemma}\label{lemma2}
For any $n\in\mathbb{N}$, we have

(1) $\varepsilon_{n}=0$ if and only if $\tau_{0}(n)\in\{0,1,\cdots,k-1\}$,

(2) $\varepsilon_{n}=1$ if and only if $\tau_{0}(n)=k $.
\end{lemma}
\begin{proof}It's only need to prove that $\varepsilon_{n}=0$ if and only if $\tau_{0}(n)\in\{0,1,\cdots,k-1\}$ by induction on $n$. The other statement has the same proof.

It is true for $n=0,1,\cdots,k$, since $\varepsilon=\varepsilon_{0}\varepsilon_{1}\varepsilon_{2}\cdots=0^{k}1\cdots$.
Assume that it's true for $m<f_{n}$ for some $n\geq1$, we need prove it's true for $f_{n}\leq m<f_{n+1}=k f_{n}+f_{n-1}$.

If $f_{n}\leq m<f_{n}+f_{n-1}$, then we have $ m-f_{n}\leq f_{n-1}$. Let $m-f_{n}=\sum_{i=0}^{n-2}m_{i}f_{i}$ be a regular expression. Hence, $m=\sum_{i=0}^{n-2}m_{i}f_{i}+f_{n}$ is the regular expression of $m$. By Corollary \ref{corollary1}, we have
\begin{eqnarray*}
\varepsilon_{m}=0&\Leftrightarrow&\varepsilon_{m-f_{n}}=0\Leftrightarrow \tau_{0}(m-f_{n})\in\{0,1,\cdots,k-1\}\\&\Leftrightarrow& m_{0}\in\{0,1,\cdots,k-1\}\Leftrightarrow \tau_{0}(m)\in\{0,1,\cdots,k-1\}.
\end{eqnarray*}

Continue this process, the result is true for $t f_{n}+f_{n-1}\leq m<(t+1)f_{n}+f_{n-1} (t=1,2,\cdots,k-1)$, which ends the proof.
\end{proof}

\begin{lemma}\label{lemma3}
Let $n=\Sigma_{i\geq0}n_{i}f_{i}$ with $n_{i}\in\{0,1,\cdots,k\}$, Assume that for $0\leq i \leq t$ if $n_{i+1}=k$, then $n_{i}=0$.
Then we have $n_{i}=\tau_{i}(n)$ for $0\leq i \leq t$.
\end{lemma}
\begin{proof} If there exists $i\in\mathbb{N}$ such that $n_{i+1}=k$, but $n_{i}\neq 0$. Let $i_{0}=\max\{i\geq0:n_{i+1}=k, n_{i}\neq 0\}$. Let $j_{0}=\max\{j\geq0:n_{i_{0}+1}=n_{i_{0}+3}=\cdots = n_{j}=k\}$. Clearly, $j_{0}\geq i_{0}+1 $,
then we have
\begin{eqnarray*}
n&=&\sum_{i=0}^{i_{0}-1}n_{i}f_{i}+\sum_{i=i_{0}}^{j_{0}}n_{i}f_{i}+\sum_{i=j_{0}+1}^{\infty}n_{i}f_{i}
\\&=&\sum_{i=0}^{i_{0}-1}n_{i}f_{i}+(n_{i_{0}}-1)f_{i_{0}}+(n_{j_{0}+1}+1)f_{j_{0}+1}+\sum_{i=j_{0}+2}^{\infty}n_{i}f_{i}
\\&=&\sum_{i\geq0}n_{i}^{\prime}f_{i}.
\end{eqnarray*}

The new expression is either regular or has a smaller maximum index $i$ with $n_{i+1}=k$, but $n_{i}\neq 0$.
and unchanged at the index less than $t$. By continuing this procedure, we finally get the regular expression of $n$, which does not change the original expression at the indices less than $t$. Hence, we have $n_{i}=\tau_{i}(n)$ for $0\leq i \leq t$.
\end{proof}

\begin{lemma}\label{lemma4}
$\forall ~i \in \mathbb{N}$, $n\geq0$, $\varepsilon_{i+f_{n}}\neq\varepsilon_{i}$
if and only if $i\equiv_{n+1}f_{n+1}-2$ or $i\equiv_{n+1}f_{n+1}-1$.
Moreover, $$\varepsilon_{i+f_{n}}-\varepsilon_{i}=\left\{\begin{array}{ll}(-1)^{n} & \textrm{if} \ i\equiv_{n+1}f_{n+1}-2, \\ (-1)^{n+1} & \textrm{if}\  i \equiv_{n+1}f_{n+1}-1.\end{array}\right.$$

\end{lemma}
\begin{proof} If $n=0$, it is easy to check that it is true. Now we assume that $n\geq1$. Hence there are two cases for discussion.

Case 1: $n\geq1$, and $\tau_{n}(i)\neq k$.
Then, we have
\begin{eqnarray*}
i+f_{n}=\sum_{j=0}^{n-1}\tau_{j}(i)f_{j}+(\tau_{n}(i)+1)f_{n}+\sum_{j=n+1}^{\infty}\tau_{j}(i)f_{j}.
\end{eqnarray*}

If $n\geq2$, then $\tau_{0}(i+f_{n})=\tau_{0}(i)$ by Lemma \ref{lemma3};

If $n=1,0\leq\tau_{n}(i)\leq k-2$, then $\tau_{0}(i+f_{n})=\tau_{0}(i)$ by Lemma \ref{lemma3};

If $n=1,\tau_{n}(i)=k-1,\tau_{0}(i)=0 $, then $\tau_{0}(i+f_{n})=\tau_{0}(i)$ by Lemma \ref{lemma3};

If $n=1,\tau_{n}(i)=k-1,\tau_{0}(i)\neq0 $, then we have
\begin{eqnarray*}
i+f_{n}=\tau_{0}(i)+k f_{1}+\sum_{j=2}^{\infty}\tau_{j}(i)f_{j}=(\tau_{0}(i)-1)+(\tau_{2}(i)+1)f_{2}+\sum_{j=3}^{\infty}\tau_{j}(i)f_{j}.
\end{eqnarray*}
By Lemma \ref{lemma3}, we have  $\tau_{0}(i+f_{n})=\tau_{0}(i)-1$. Hence, by Lemma \ref{lemma2},
if $1\leq\tau_{0}(i)\leq k-1$, we have $\varepsilon_{i+f_{n}}=\varepsilon_{i}$;
if $\tau_{0}(i)=k$, we have $\varepsilon_{i+f_{n}}\neq\varepsilon_{i}$.

Hence, in the case 1: $n\geq1$, and $\tau_{n}(i)\neq k$. We have
\begin{eqnarray*}
\varepsilon_{i+f_{n}}\neq\varepsilon_{i} &\Leftrightarrow & n=1,\tau_{0}(i)=k,\tau_{n}(i)=k-1, \\&\Leftrightarrow & i\equiv_{n+1}f_{n+1}-2~(n=1).
\end{eqnarray*}

Case 2: $n\geq1$, and $\tau_{n}(i)=k$.
Let $j=\min\{t\geq0:\tau_{n}(i)=\tau_{n-2}(i)=\cdots=\tau_{t}(i)=k\}$. Clearly, $\tau_{j-2}(i)\leq k-1$. By the equation $(k+1)f_{n}=f_{n+1}+(k-1)f_{n-1}+f_{n-2}$, we have
\begin{eqnarray*}
i+f_{n}&=&\sum_{s=0}^{j-3} \tau_{s}(i)f_{s}+(\tau_{j-2}(i)+1)f_{j-2}+(k-1)f_{j-1}\\&+&k f_{j+1}+\cdots+k f_{n-1}+(\tau_{n+1}(i)+1)f_{n+1}+\sum_{s=n+2}^{\infty}\tau_{s}(i)f_{s}.
\end{eqnarray*}
where the first term on the right-hand side vanishes if $j = 0, 1, 2$. Hence

If $j\geq4$, by Lemma \ref{lemma3}, we have $\tau_{0}(i+f_{n})=\tau_{0}(i)$;

If $j=3$, then $\tau_{1}(i)\leq k-1$ and $i+f_{n}=\tau_{0}(i)f_{0}+(\tau_{1}(i)+1)f_{1}+(k-1)f_{2}+k f_{4}+\cdots+k f_{n-1}+(\tau_{n+1}(i)+1)f_{n+1}+\sum_{s=n+2}^{\infty}\tau_{s}(i)f_{s}$.  Hence there are following subcases;

\quad If $\tau_{0}(i)=0$, by Lemma \ref{lemma3}, we have $\tau_{0}(i+f_{n})=\tau_{0}(i)$;

\quad If $\tau_{0}(i)\neq0, \tau_{1}(i)\leq k-2$, by Lemma \ref{lemma3}, we have $\tau_{0}(i+f_{n})=\tau_{0}(i)$;

\quad If $\tau_{0}(i)\neq0, \tau_{1}(i)=k-1$, hence if $\tau_{0}(i)=k$, then $\varepsilon_{i+f_{n}}\neq\varepsilon_{i}$, and if $\tau_{0}(i)\neq k$, then $\varepsilon_{i+f_{n}}=\varepsilon_{i}$.

If $j=2$, then $\tau_{0}(i)\leq k-1$ and $i+f_{n}=(\tau_{0}(i)+1)f_{0}+(k-1)f_{1}+k f_{3}+\cdots+k f_{n-1}+(\tau_{n+1}(i)+1)f_{n+1}+\sum_{s=n+2}^{\infty}\tau_{s}(i)f_{s}.$
By Lemma \ref{lemma3}, we have $\tau_{0}(i+f_{n})=\tau_{0}(i)+1$. So if $\tau_{0}(i)=k-1$, then $\varepsilon_{i+f_{n}}\neq\varepsilon_{i}$;

If $j=1$, then $\tau_{0}(i)=0$ and $i+f_{n}=k f_{0}+k f_{2}+\cdots+k f_{n-1}+(\tau_{n+1}(i)+1)f_{n+1}+\sum_{s=n+2}^{\infty}\tau_{s}(i)f_{s}.$
Hence $\tau_{0}(i+f_{n})=k$, so we have $\varepsilon_{i+f_{n}}\neq\varepsilon_{i}$;

If $j=0$, then $\tau_{0}(i)=k$ and $i+f_{n}=k f_{1}+k f_{3}+\cdots+k f_{n-1}+(\tau_{n+1}(i)+1)f_{n+1}+\sum_{s=n+2}^{\infty}\tau_{s}(i)f_{s}.$
Hence, $\varepsilon_{i+f_{n}}\neq\varepsilon_{i}$.

Combining all discussions of the second case, we obtain the first part.

The second part follows from Lemma \ref{lemma2} directly.
\end{proof}

\section{Proofs of Theorem \ref{theorem1} and Proposition \ref{p1}}
To prove Theorem \ref{theorem1}, we need following lemma. We state it as follow without proof.
\begin{lemma}[Adamaczewski and Rivoal \cite{AR09}]\label{lemma6}
Let $\xi,\alpha,\beta,\gamma\in\mathbb{R}$ and $\alpha\leq\beta$, $\gamma>1$. Assume that there exist positive real numbers $c_{0},c_{1}\leq c_{2}$ and a sequence $(\frac{p_{n}}{q_{n}})_{n\geq1}$ of
rational numbers such that
\begin{equation*}
q_{n}<q_{n+1}\leq c_{0}q_{n}^{\gamma},(n\geq1),
\end{equation*}
and
\begin{equation*}
\frac{c_{1}}{q_{n}^{1+\beta}}\leq|\xi-\frac{p_{n}}{q_{n}}|\leq\frac{c_{2}}{q_{n}^{1+\alpha}},(n\geq1).
\end{equation*}
Then $\mu(\xi)\leq(1+\beta)\frac{\gamma}{\alpha}$.
\end{lemma}

%Let $\frac{p}{q}$ be a reduced rational number whose denominator is large enough.Let $n\geq2$ be the uniquely integer such that
%$q_{n-1}<(2c_{2}q)^{\frac{1}{\theta}}\leq q_{n}$.

Now, we are going to prove Theorem \ref{theorem1}.

{\bf Proof of Theorem \ref{theorem1}}
For any integer $k\geq1$, let $\varphi_{k}(z)=\sum_{i\geq0}\varepsilon_{i}z^{i}$ be the generating function of the sequence $\varepsilon$.
%Then, we have following lemma which has been described in \cite{Tamura99,KTW99}.
%\begin{lemma}\label{lemma5}
For any integer $n\geq2$, define integer polynomials  $P_{n}(z)=\varepsilon_{0}+\varepsilon_{1}z+\cdots+\varepsilon_{f_{n}-1}z^{f_{n}-1}$,
$Q_{n}(z)=1-z^{f_{n}}$, then, by Corollary \ref{corollary1},
\begin{eqnarray*}
 Q_{n}(z)\varphi_{k}(z)-P_{n}(z)&=&\sum_{i\geq0}(\varepsilon_{f_{n}+i}-\varepsilon_{i})z^{f_{n}+i} \\
   &=& \sum_{i\geq f_{n+1}-2}(\varepsilon_{f_{n}+i}-\varepsilon_{i})z^{f_{n}+i}.
\end{eqnarray*}
%which implies that $\frac{P_{n}(z)}{Q_{n}(z)}$ is a Pad\'{e} approximant of $\varphi_{k}(z)$.
%\end{proof}

%Let $n\geq2,k\geq1$ be integers, by Lemma \ref{lemma5},
%\begin{eqnarray*}
%Q_{n}(z)\varphi_{k}(z)-P_{n}(z)=(\varepsilon_{f_{n}+f_{n+1}-2}-\varepsilon_{f_{n+1}-2})z^{f_{n}+f_{n+1}-2}+\cdots.
%\end{eqnarray*}

Set $\delta_{n}(z):=\varphi_{k}(z)-\frac{P_{n}(z)}{Q_{n}(z)}$ , by Lemma \ref{lemma4}, then
\begin{eqnarray*}
\delta_{n}(z)&=&\frac{1}{1-z^{f_{n}}}\cdot z^{f_{n}+f_{n+1}-2}[(\varepsilon_{f_{n}+f_{n+1}-2}-\varepsilon_{f_{n+1}-2})+(\varepsilon_{f_{n}+f_{n+1}-1}-\varepsilon_{f_{n+1}-1})z+\cdots]
\\&=&\frac{1}{1-z^{f_{n}}}\cdot z^{f_{n}+f_{n+1}-2}[(-1)^{n}+(-1)^{n-1}z+(-1)^{n}z^{f_{n+1}}+(-1)^{n-1}z^{f_{n+1}+1}+\cdots]
\\&=&\frac{1}{1-z^{f_{n}}}\cdot z^{f_{n}+f_{n+1}-2}[(-1)^{n}+(-1)^{n-1}z]\sum_{\begin {subarray}{c} j\geq0,\\\tau_{0}(j)\neq k.\end {subarray}}z^{\sum_{i\geq0}\tau_{i}(j)f_{n+1+i}}.
\end{eqnarray*}

Let $b\geq2$ be an integer. Taking $z=\frac{1}{b}$, we have
\begin{equation*}
  \left|\delta_{n}(\frac{1}{b})\right|=\frac{b^{f_{n}}}{b^{f_{n}}-1}\cdot\frac{1}{b^{f_{n}+f_{n+1}-2}}\cdot\left|(-1)^{n}+(-1)^{n-1}\frac{1}{b}\right|\cdot\left|\sum_{\begin {subarray}{c} j\geq0,\\\tau_{0}(j)\neq k\end {subarray}}b^{-\sum_{i\geq0}\tau_{i}(j)f_{n+1+i}}\right|.
\end{equation*}

Hence,
\begin{equation*}
  \left|\delta_{n}(\frac{1}{b})\right|\leq\frac{b^{f_{n}}}{b^{f_{n}}-1}\cdot\frac{1}{b^{f_{n}+f_{n+1}-2}}\cdot(1-\frac{1}{b})\sum_{i\geq0}\frac{1}{b^{i}}
=\frac{1}{b^{f_{n}}-1}\cdot\frac{1}{b^{f_{n+1}-2}}
\end{equation*}
and
\begin{equation*}
   \left|\delta_{n}(\frac{1}{b})\right|\geq\frac{b^{f_{n}}}{b^{f_{n}}-1}\cdot\frac{1}{b^{f_{n}+f_{n+1}-2}}\cdot(1-\frac{1}{b})
=\frac{1}{b^{f_{n}}-1}\cdot\frac{1}{b^{f_{n+1}-2}}\cdot\frac{b-1}{b}.
\end{equation*}

Thus, for any integer $b\geq2$,
\begin{equation}\label{formula3}
\frac{1}{b^{f_{n}}-1}\cdot\frac{1}{b^{f_{n+1}-2}}\cdot\frac{b-1}{b}\leq\left|\delta_{n}(\frac{1}{b})\right|\leq\frac{1}{b^{f_{n}}-1}\cdot\frac{1}{b^{f_{n+1}-2}}.
\end{equation}

Define integers
\begin{eqnarray*}
p_{n}=b^{f_{n}}\cdot P_{n}(\frac{1}{b})\in\mathbb{Z},
\end{eqnarray*}
\begin{eqnarray*}
q_{n}=b^{f_{n}}\cdot Q_{n}(\frac{1}{b})=b^{f_{n}}-1\in\mathbb{Z}.
\end{eqnarray*}

By Formula (\ref{formula3}), there exist positive numbers $c_{0},c_{1}\leq c_{2}$, depending only on $b$, such that
\begin{equation}\label{formula4}
q_{n}<q_{n+1}\leq c_{0}q_{n}^{\frac{f_{n+1}}{f_{n}}},(n\geq1),
\end{equation}

and
\begin{equation}\label{formula5}
\frac{c_{1}}{q_{n}^{1+\frac{f_{n+1}}{f_{n}}}}\leq|\varphi_{k}(\frac{1}{b})-\frac{p_{n}}{q_{n}}|\leq \frac{c_{2}}{q_{n}^{1+\frac{f_{n+1}}{f_{n}}}},(n\geq1).
\end{equation}

Let $\theta:=\lim_{n\rightarrow\infty}\theta_{n}=\lim_{n\rightarrow\infty}\frac{f_{n+1}}{f_{n}}=\frac{k+\sqrt{k^{2}+4}}{2}$. Then for any real number $\epsilon\geq0$, by Formula (\ref{formula4}) and Formula (\ref{formula5}), there exists a large number $n_{0}(\varepsilon)$, such that

\begin{eqnarray*}
q_{n}<q_{n+1}\leq c_{0}q_{n}^{\theta+\epsilon},(n\geq n_{0}(\epsilon));
\end{eqnarray*}

and
\begin{eqnarray*}
\frac{c_{1}}{q_{n}^{\theta+\epsilon}}\leq|\varphi_{k}(\frac{1}{b})-\frac{p_{n}}{q_{n}}|\leq \frac{c_{2}}{q_{n}^{\theta-\epsilon}},(n\geq n_{0}(\epsilon)).
\end{eqnarray*}

By Lemma \ref{lemma6}, we have $\theta-\epsilon\leq\mu(\varphi_{k}(\frac{1}{b}))\leq (\theta+\epsilon)\frac{\theta+\epsilon}{\theta-\epsilon}$.
Let $\epsilon\rightarrow0$, we have $\mu(\varphi_{k}(\frac{1}{b}))=1+\theta$ for any integer $b\geq2$, which completes this proof.

Then we are going to prove Proposition \ref{p1}.

{\bf Proof of Proposition \ref{p1}}
Let $\mathbf{v}=\{v_i\}_{i\geq0}=\mathbf{u}\times S(\mathbf{u})=\{[u_{i},u_{i+1}]\}_{i\geq0}$.
Note that the value of $[u_{i},u_{i+1}]$ is determined by the block $``u_{i}u_{i+1}"$, but there are exactly $3$ blocks with length $2$ in any Sturmian sequence. Without loss of generality, assume
\begin{equation*}
  \{u_iu_{i+1}:i\geq0\}=\{01,10,11\}.
\end{equation*}
Then there exist $a_{i}\in\mathbb{Q}~(0\leq i\leq 2)$ such that
\begin{equation*}
    [u_{i},u_{i+1}]=a_{0}u_{i}+a_{1}u_{i+1}+a_{2}.
\end{equation*}

In fact, $a_0=[1,1]-[0,1],a_1=[1,1]-[1,0],a_2=[0,1]+[1,0]-[1,1]$. Since $\mathbf{v}$ is non-eventually periodic, at least one of $a_0,a_1$ is not zero.
Thus, we have
\begin{eqnarray*}
  \xi_{\mathbf{v},b} &=& \sum_{i\geq0}v_{i}\cdot\frac{1}{b^{i}} \\
  &=& \sum_{i\geq0}(a_{0}u_{i}+a_{1}u_{i+1}+a_{2})\cdot\frac{1}{b^{i}} \\
  &=& a_{0}\sum_{i\geq0}\frac{u_{i}}{b^{i}}+ a_{1}\sum_{i\geq0}\frac{u_{i+1}}{b^{i}}+a_{2}\frac{b}{b-1}\\
  &=& a_{0}\xi_{\mathbf{u},b}+a_{1}b(\xi_{\mathbf{u},b}-u_{0})+a_{2}\frac{b}{b-1}.
\end{eqnarray*}

By the definition of irrationality exponent, we have, for any integer $b\geq2$, $\mu(\xi_{\mathbf{u}\times S(\mathbf{u}),b})=\mu(\xi_{\mathbf{u},b})=a$, which completes this proof.

{\bf  Proof of Corollary \ref{c2}}
By the definition of the difference operation, we see that for $k\geq1$, $\Delta^{k}(\mathbf{u})=\left(\sum_{i=0}^{k}{k \choose i}u(n+i) \right)\mod2$.
Hence, the value of $\Delta^{k}(\mathbf{u})$ is determined by the block $``u(n)u(n+1)\cdots u(n+k)"$, and there are exactly $k+2$ blocks with length $k+1$ in any Sturmian sequence.
Hence, by a same method of Proposition \ref{p1}, we have for any integer $k\geq1$, $\mu(\xi_{\Delta^{k}(\mathbf{u}),b})=\mu(\xi_{\mathbf{u},b})=a$ for any integer $b\geq2$.

\noindent\textbf{Acknowledgements. }
The authors gratefully thank Professor Jeffrey Shallit for his helpful suggestions.


\medskip


\medskip

\begin{thebibliography}{8}
\bibitem {AA07} Adamczewski B, Allouche J P, Reversals and plaindromes in continued fractions, Theoret. Comput. Sci. 2007, \textbf{380}: 220-237.

\bibitem {AY07} Adamczewski B, Bugeaud Y, On the complexity of algebraic numbers,I.Expansions in integer bases, Ann. of Math. 2007, \textbf{165}(2): 547-565.

\bibitem {AC06} Adamczewski B, Cassaigne J, Diophantine properties of real numbers generated by finite automata, Compos. Math. 2006, \textbf{142}: 1351-1372.

\bibitem {AR09} Adamczewski B, Rivoal T, Irrationality measures for some automatic real numbers, Math. Proc. Cambridge Philos. Soc. 2009, \textbf{147}: 659-678.

\bibitem {AD77} Adams W W, Davison J L, A remarkable class of continued fractions, Proc. Amer. Math. Soc. 1977, \textbf{65}: 194-198.

\bibitem {AS03} Allouche J P, Shallit J, Automatic Sequence, Theory, Applications, Generalizations, Cambridge University Press, Cambridge, 2003.

%\bibitem{APWW98} J.-P. Allouche, J. Peyri\`{e}re, Z.-X. Wen and Z.-Y. Wen, Hankel determinants of the Thue-Morse sequence, Ann. Inst. Fourier (Grenoble) 48 (1998) 1-27.

%\bibitem {BG96} G.A.Jr. Baker, P. Graves-Morris, Pad\'{e} approximants, Second edition. Encyclopedia of Mathematics and its Applications, 59. Cambridge University Press, Cambridge, 1996. xiv+746 pp.

%\bibitem {B80} C. Brezinski, Pad\'{e}-type Approximation and General Orthogonal Polynomials, Internaltional Series of Numerical Mathematics, Vol. 50, Birkh\"{a}user Verlag, 1980.

\bibitem {B26} B\"{o}hmer P E, \"{U}ber die Transzendenz gewisser dyadischer Br\"{u}che, Math. Annalen. 1926, \textbf{96}: 367-377.

\bibitem {Bugeaud11} Bugeaud Y, On the rational approximation of The Thue-Morse-Mahler numbers,  Ann. Institut Fourier, 2011, \textbf{61}: 2065-2076.

%\bibitem {Bugeaud08} Y. Bugeaud, Diophantine approximation and Cantor sets,  Math. Ann. 341 (2008) 677-684.

\bibitem {BKS11} Bugeaud Y, Krieger D, Shallit J, Morphic and Automatic Words: Maximal Blocks and Diophantine Approximation, Acta Arith. 2011, \textbf{149}: 181-199.

%\bibitem {C68} A. Cobham, Uniform tag sequences, Math. Systems Theorey. 6 (1972) 164-192.

\bibitem {BS94} Bullett S, Sentenac P, Ordered orbits of the shift, square roots, and the devil's staircase, Proc. Cambridge Phil. Soc. 1994, \textbf{115}: 451-481.

\bibitem {CV12} Coons M, Vrbik P, An irrationality measure for regular paperfolding numbers, J. Integer Seq. 2012, \textbf{15}: 1-10.

\bibitem {C12} Coons M, On the rational approximation of the sum of the reciprocals of the Fermat numbers, Ramanujan J. 2013, \textbf{30}(1): 39-65.

\bibitem {Danilov72}  Danilov L V, Some classes of transcendental numbers, Mat. Zametki. 1972, \textbf{12}: 149-154.

\bibitem {Davison77} Davison J L, A series and its associated continued fraction, Proc. Amer. Math. Soc. 1977, \textbf{63}: 29-32.

\bibitem {GWW14}Guo Y J, Wen Z X, Wu W, On the irrationality exponent of the regular paperfolding numbers, Linear Algebra Appl. 2014, \textbf{446}: 237-264.

\bibitem {GW14} Guo Y J, Wen Z X, Automaticity of the Hankel determinants of difference sequence of the Thue-Morse sequence and rational approximation, Theoret. Comput. Sci. 2014, \textbf{552}: 1-12.

\bibitem {S99} Shiu P, A function from Diophantine approximations, Publ. Inst. Math.(Beograd) 1999, \textbf{65}: 52-62.

\bibitem {Tamura99} Tamura J I, Pad\'{e} approximation for infinite words generated by certain substitutions, and Hankel determinants. In: Number Theory and Its Applications, K.Gy\"{o}ry and S.Kanemitsu (eds.), Kluwer Academic Publishers. 1999: 309-346.

\bibitem {KTW99} Kamae T, Tamura J I, Wen Z Y, Hankel determinants for the Fibonacci word and Pad\'{e} approximation, Acta Arith. 1999, \textbf{2}: 123-161.

\bibitem {WW93} Wen Z X, Wen Z Y, Some studies of factors of infinite words generated by invertible substitution, A. Barlotti, M. Delest, R. Pinzani (Eds.), Proc. 5th Conf. Formal Power Series and Algebraic Combinatorics 1993: 455-466.

%\bibitem{WW94} Z.-X. Wen, Z.-Y. Wen, Local isomorphism of the invertible substitutions. C.R.Acad.Sci.Paris S¨¦r. I Math., 318 (1994), 299¨C304.

\bibitem {WW94} Wen Z X, Wen Z Y, Some properties of the singular words of the Fibonacci word, European J. Combin. 1994, \textbf{15}: 587-598.

\bibitem {WWW02}Wen Z X, Wen Z Y, Wu J, Invertible substitutions and local isomorphisms, C. R. Acad. Sci. Paris S\'{e}r. I Math. 2002, \textbf{334}: 629-634.

\end{thebibliography}
\end{document}